\documentclass[11pt]{article} 

\usepackage[utf8]{inputenc} 
\usepackage{amsmath, amsthm, amssymb}

\usepackage{geometry} 
\usepackage{graphicx} 

\usepackage[usenames,dvipsnames,svgnames,table]{xcolor}

\usepackage{fancyhdr} 
\newcommand{\su}{\textswab{s:u(2)}}
\usepackage{sectsty}
\allsectionsfont{\sffamily\mdseries\upshape} 

\usepackage{verbatim} 

\usepackage{yfonts}[1998/10/03]

\newtheorem{thm}{Theorem}[section]
\newtheorem{cor}[thm]{Corollary}
\newtheorem{lem}[thm]{Lemma}

\newtheorem{prop}[thm]{Proposition}

\theoremstyle{definition}

\theoremstyle{definition}
\newtheorem{defn}[thm]{Definition}
\theoremstyle{remark}
\newtheorem{rem}[thm]{Remark}

\newtheorem{thm*}[thm]{Theorem}

\DeclareMathOperator{\spann}{span}
\DeclareMathOperator{\sgn}{sgn}
\DeclareMathOperator{\tr}{Tr}

\title{\sc{$SU_{q}(2)$ Representation Theory}}

\author{Olof Giselsson}

\begin{document}

\maketitle


\section{$SU(2)$ as a Lie group}
We will start this by looking at the Lie group $SU(2),$ and then to show how you get the irreducible representations of the corresponding Lie algebra $\su.$
This is to, hopefully, give some background to the reader (and the autor) so that the similarities with the general CQG (compact quantum group) case shows a bit. If this is familiar stuff, be my guest and skip it.
\\

Recall that the complex Lie group $SU(2)$ is a matrix subgroup of $M_{2}(\mathbb{C}),$ consisting of all matrices $$\left[ \begin{array}{ccc}
			\alpha & -\bar{\beta}  \\
			\beta & \bar{\alpha}  
			\end{array} \right]$$ such that $|\alpha|^{2}+|\beta|^{2}=1.$ It is not hard to see that $SU(2)\cong S^{3}$ as Lie groups, where $S^{3}$ (unit ball in $\mathbb{R}^{4}$) has the quaternion group structure, and hence $SU(2)$ is a compact Lie group. The complexified Lie algebra associated to $SU(2)$ will be denoted by $\su$ and is a $3$ dimensional vector space over $\mathbb{C}$ generated by three vectors ${h,e,f}$ subject to the conditions 
\begin{equation}\label{su1}
\begin{array}{ccc}
[h,e]=2 e & [f,h]=2f & [f,e]= h
\end{array}.
\end{equation}
When determining the irreducible representations of $SU(2)$ (and Lie groups in general) it is well advised to look at the irreducible representations for the Lie algebra instead, since these are in one to one correspondence with each other (in fact they are the 'same', as the Lie algebra is seen as "infinitesimal" actions by the group). An representation of a Lie algebra $A$ on a complex vector space $V$ is a linear map $\phi:A\rightarrow GL(V),$ such that $\phi([a,b])=[\phi(a),\phi(b)]$ for all $a,b\in A,$ where the first bracket is the Lie bracket in $A$ and the second is the commutator in $GL(V).$
\\

As $SU(2)$ is compact, we know that every irreducible representation is finite dimensional. Hence we need to analyze the situation when we have $f,e,h\in GL(V)$ satisfying~\eqref{su1} and these are irreducible (no nontrivial subspace is invariant for all of them). Fortunately, this is not very hard. As $V$ is complex, we know that $h$ has at least one eigenvector $v$ with eigenvalue $\lambda.$ We claim that, if non-zero, also $ev$ and $fv$ are eigenvectors for $h,$ with eigenvalue $\lambda +2$ and $\lambda-2$ respectively. Indeed, a calculation gives $$
\begin{array}{ccc}
h(ev)=ehv+[h,e]v=\lambda ev+2ev=(\lambda+2)ev \\
h(fv)=fhv-[f,h]v=\lambda fv-2fv=(\lambda-2)fv
\end{array}.$$ 	
If $ev$ and $fv$ are both zero, then $\{v\}$ is an invariant subspace for $e,f,h$ and hence $\{v\}=V.$ If, say, $ev\neq 0,$ consider the set of vectors $\{v,ev,e^{2}v,....\}.$ If $e^{j}v$ and $e^{k}v$ are both non-zero and $j\neq k,$ then they are both eigenvectors of $h,$ but with different eigenvalues (by induction from what we proved above). Thus the non-zero vectors in $\{v,ev,e^{2}v,...\}$ are linear independent and since $V$ is finite dimensional, there exists a $k\in\mathbb{N},$ such that $e^{k}v=0,$ but $e^{j}v\neq 0$ for $0\leq j\leq k-1.$ Let $w:=e^{k-1}v,$ then by the same argument as before, there is a $n\in\mathbb{N},$ such that the set $\{w,fw,...f^{n}w\}$ consists of non-zero linear independent eigenvectors for $h$ and such that $f^{n+1}w=0.$
\begin{lem}
The subspace $W$ generated by the set $\{w,fw,...f^{n}w\}$ is invariant under $f,e,h$ and hence $W=V.$
\end{lem}						
\begin{proof}
$W$ is trivially invariant under $f$ and $h.$ So it remains to show that it is invariant under $e.$
Let $W_{k}$ be the subspace generated by the subset $\{w,fw,...f^{k}w\},$ we prove by induction that $W_{k}$ is invariant under $e.$ 
$W_{0}=\{w\}$ and $ew=0,$ so in this case it is clear. Assume not that it is true for $k-1,$ then every $x\in W_{k}$ can be written as $\alpha f^{k}w+x_{k-1}$ with $\alpha\in\mathbb{C}$ and $x_{k-1}\in W_{k-1}.$ We now get 
$$
e x=e(\alpha f^{k}w+x_{k-1})=\alpha e f^{k} w+e x_{k-1}=\alpha  f 
e f^{k-1} w+ \alpha [ f, e] f^{k-1} w+e x_{k-1}=
$$
$$\alpha  f e f^{k-1} w+\alpha h f^{k-1} w+e x_{k-1}$$
and all these terms are in $W_{k}$ by induction.
\end{proof}
Let $\lambda$ be the eigenvalue $hw=\lambda w.$ If we write up $h$ in the basis $\{w,fw,...f^{n}w\},$ then by the earlier results, we find that $h$ is diagonal with the numbers $\{\lambda,\lambda-2,...,\lambda-2n\}$ as the diagonal entries. As $h$ is a commutant, it must have zero trace and hence
$$0=\tr{h}=\sum_{k=0}^{n}(\lambda-2k)=(n+1)\lambda-(n+1) (n)$$
from which we deduce $\lambda=n.$ We see now that all the eigenspaces for $h$ is of dimension $1$ and thus since both $ef(f^{k}w)$ and $fe(f^{k}w),$ if non-zero, is in the same eigenspace of $h$ as $f^{k}w$ for all $k,$ we deduce that both $ef$ and $fe$ are diagonal in the $\{w,fw,...f^{n}w\}$ basis.
\begin{lem}
For all $0\leq k\leq n,$ we have $(fe)f^{k}w=-(n-k+1)kf^{k}w$ and $(ef)f^{k}w=-(n-k)(k+1)f^{k}w.$
\end{lem}
\begin{proof}
We prove the statement for both $ef$ and $fe$ at the same time by induction on $k.$
For $k=0,$ we have $(fe)w=0$ and $(ef)w=[e,f]w+(fe)w=-n w,$ so the formula holds.
Assume now $k\geq 0$ and it holds for integers less than $k,$ then
$$(ef)f^{k}w=[e,f]f^{k}w+fef^{k}w=-h f^{k}w+f(ef)f^{k-1}w=$$
$$-(n-2k)f^{k}w-(n-k+1)kf^{k}w=-(n-2k+(n-k+1)k)f^{k}w=$$
$$-(n-k)(k+1)f^{k}w$$
and
$$(fe)f^{k}w=f(ef)f^{k-1}w=-(n-k+1)k f^{k}w. $$
\end{proof}

We can from this deduce that $e$ acts as $f^{k}w\mapsto -(n-k+1)k f^{k-1}w.$ If we make a new basis $w_{k}=\frac{1}{k!}f^{k}w$ (set $w_{-1}=0$),
then in this this basis we get 
\begin{equation}\label{22}
\begin{array}{ccc}
e w_{k}=-(n-k+1)w_{k-1}\\
f w_{k}=(k+1)w_{k+1}\\
hw_{k}=(n-k)w_{k}.
\end{array}
\end{equation}
This classifies the irreducible representations of $\su$ and hence also the irreducible representations of $SU(2).$ The next theorem is a classic Lie group result.
\begin{thm}
For every integer $n\in\mathbb{N},$ there is up to equivalence a unique irreducible representation $\pi_{n}:SU(2)\rightarrow V_{n}$ of dimension $n+1.$
The corresponding representation of $\su$ acts as in~\eqref{22}.
\end{thm}

The idea of the general representation theory for $SU_{q}(2)$ below, is to find a system similar to $\su$ whose irreducible representations can also be analyzed in complete way and which are in a one to one correspondence with the irreducible representations of $SU_{q}(2).$

\section{$SU_{q}(2)$}
We define the compact quantum group (CQG) $SU_{q}(2).$ The exposition here is a slightly reworked version of~\cite{wor}.
\begin{defn}
Let $q\in [0,1]$ and $\alpha,\gamma$ be elements in a unital $*$-algebra $A$ such that 
\begin{equation}\label{mm}
\left[\begin{array}{ccc}
\alpha & -q \gamma^{*}  \\
\gamma & \alpha^{*}
\end{array}\right]\in M_{2}(A)\text{ is unitary.}
\end{equation}
 We define $SU_{q}(2)$ to be the universal $C^{*}$-algebra generated $\alpha,\gamma$ satisfying this relation.
\end{defn}

"More" concretely, this $C^{*}$-algebra can be constructed by first taking $I$ to be the index set of all triples $\{\alpha_{\lambda},\gamma_{\lambda},\mathcal{H}_{\lambda}\}$ such that $\mathcal{H}_{\lambda}$ is a Hilbert space, $\alpha_{\lambda},\gamma_{\lambda}\in B(\mathcal{H}_{\lambda})$ and the matrix~\eqref{mm} with $\alpha=\alpha_{\lambda},\gamma=\gamma_{\lambda}$ is unitary. Take now $\alpha'=\prod_{\lambda\in I}(\alpha_{\lambda})$ and $\gamma'=\prod_{\lambda\in I}(\gamma_{\lambda})$ as elements in $ \prod_{\lambda\in I}B(\mathcal{H}_{\lambda})$ taken as a $C^{*}$-algebra with the product norm $\sup_{\lambda \in I}\|.\|_{\lambda}.$ Note that as~\eqref{mm} is unitary, both $\alpha_{\lambda}$ and $\gamma_{\lambda}$ must have norm $\leq 1$ and hence $\prod_{\lambda\in I}(\alpha_{\lambda})$ and $\prod_{\lambda\in I}(\gamma_{\lambda})$ actually defines elements in $\prod_{\lambda\in I}B(\mathcal{H}_{\lambda}),$ also with norm $\leq 1.$ We can now define $SU_{q}(2)'$ to be the $C^{*}$-algebra in $\prod_{\lambda\in I}B(\mathcal{H}_{\lambda})$ generated by $\alpha'$ and $\gamma'$. By universality of $SU_{q}(2),$ we get a surjective homomorphism $\phi:SU_{q}(2)\rightarrow SU_{q}(2)'$  such that $\phi(\alpha)= \alpha'$ and $\phi(\gamma)=\gamma'.$ However, by properties of the product and the Gelfand-Newmark construction, we get a surjective homomorphism $\psi:SU_{q}(2)'\rightarrow SU_{q}(2)$ such that $\psi(\alpha')=\alpha$ and $\psi(\gamma')=\gamma.$ Thus $\phi \circ \psi= id$ and $\psi\circ \phi =id$ and hence $SU_{q}(2)\cong SU_{q}(2)'$
\\

The condition that~\eqref{mm} is unitary is easily seen to be equivalent with the following set of equations
\begin{equation}\label{nn}
\begin{array}{ccc}
\alpha^{*} \alpha+\gamma^{*} \gamma=I, & \alpha \alpha^{*}+ q^{2}\gamma^{*} \gamma=I, & \gamma^{*}\gamma=\gamma \gamma^{*},\\
\alpha \gamma=q \gamma \alpha, & \alpha \gamma^{*}=q \gamma^{*} \alpha.
\end{array}
\end{equation}

\begin{prop}
$SU_{q}(2)$ can be given the structure of a compact quantum group.
The comultiplication $\Delta:SU_{q}(2)\rightarrow SU_{q}(2)\otimes SU_{q}(2)$ is acts on the generators $\alpha,\gamma$ as 
\begin{equation}
\begin{array}{ccc}
\Delta(\alpha)=\alpha\otimes \alpha-q \gamma^{*}\otimes \gamma, & \Delta (\gamma)= \gamma\otimes \alpha+ \alpha^{*}\otimes \gamma.
\end{array}
\end{equation}
\end{prop}
\begin{proof}
 First, we need to properly define the comultiplication. Set $\alpha':=\alpha\otimes \alpha-q \gamma^{*}\otimes \gamma$ and $\gamma':=\gamma\otimes \alpha+ \alpha^{*}\otimes \gamma.$ Now let $U:=\left[\begin{array}{ccc}
\alpha & -q \gamma^{*}  \\
\gamma & \alpha^{*}
\end{array}\right]\in M_{2}(SU_{q}(2))$ and $W:=\left[\begin{array}{ccc}
\alpha' & -q \gamma'^{*}  \\
\gamma' & \alpha'^{*}
\end{array}\right]\in M_{2}(SU_{q}(2)\otimes SU_{q}(2)).$ A quick calculation gives that $U_{1,2}U_{1,3}=W$ and as $U$ is unitary, so must also $W$ be. The universal property of $SU_{q}(2)$ now gives a homomorphism $\Delta:SU_{q}(2)\rightarrow SU_{q}(2)\otimes SU_{q}(2).$ The coassociativity $(\Delta\otimes I)\Delta=(I\otimes \Delta)\Delta$ follows from 
$$(\iota\otimes \Delta\otimes I)(\iota\otimes \Delta)U=U_{1,2}U_{1,3}U_{1,4}$$
$$(\iota\otimes I\otimes \Delta)(\iota\otimes \Delta)U=U_{1,2}U_{1,3}U_{1,4}$$
that proves coassociativity for $\alpha,\gamma$ and hence for for the whole algebra, since they are generating it. 
\\

The only thing left to prove is to show cocancellation (recall that this means that the linear spans of the sets $\Delta(SU_{q}(2))(I\otimes SU_{q}(2))$ and  $\Delta(SU_{q}(2))( SU_{q}(2)\otimes I)$ are both dense in $SU_{q}(2)\otimes SU_{q}(2)$). We prove the denseness of the linear span of $\Delta(SU_{q}(2))(I\otimes SU_{q}(2)),$ the rest of the proof is similar. Notice that we only need to prove $b\otimes I\in \spann(\Delta(SU_{q}(2))(I\otimes SU_{q}(2)))$ for $b$ in a dense subset of $SU_{q}(2).$ First we prove that if $b\otimes I,g\otimes I \in \spann(\Delta(SU_{q}(2))(I\otimes SU_{q}(2))),$ then also $bg\otimes I\in \spann(\Delta(SU_{q}(2))(I\otimes SU_{q}(2))).$ Let $b\otimes I=\sum \Delta(a_{k})(I\otimes c_{k})$ and $g\otimes I=\sum \Delta(f_{k})(I\otimes h_{k}) $ then 
$$bg\otimes I=(b\otimes I)(g\otimes I)=\sum \Delta(a_{k})(I\otimes c_{k})(g\otimes I)=\sum \Delta(a_{k})(g\otimes I)(I\otimes c_{k})=$$
$$\sum \Delta(a_{k}f_{j})(I\otimes h_{j}c_{k}).$$
Thus we need now only to show that $\alpha\otimes I,\gamma\otimes I\in  \spann(\Delta(SU_{q}(2))(I\otimes SU_{q}(2))).$
Using the equations~\eqref{nn}, we get 
$$\alpha\otimes I=\alpha\otimes \alpha \alpha^{*}+q^{2}\alpha\otimes \gamma^{*}\gamma=\alpha\otimes \alpha \alpha^{*}+q^{2}\alpha\otimes \gamma^{*}\gamma+q(\gamma^{*}\otimes \gamma-\gamma^{*}\otimes \gamma)(I\otimes \alpha^{*})= $$
$$(\alpha\otimes \alpha-q \gamma^{*}\otimes \gamma)(I\otimes \alpha^{*})+q^{2}\alpha\otimes \gamma^{*}c+q \gamma^{*}\otimes  \gamma \alpha^{*}=$$
$$(\alpha\otimes \alpha-q \gamma^{*}\otimes \gamma)(I\otimes \alpha^{*})+q^{2}(\alpha\otimes \gamma^{*}+\gamma^{*}\otimes \alpha^{*})(I\otimes \gamma)=$$
$$\Delta(\alpha)(I\otimes \alpha^{*})+q^{2}\Delta(\gamma^{*})(I\otimes \gamma)$$
and also 
$$\gamma\otimes I=\gamma\otimes \alpha \alpha^{*}+q^{2}\gamma\otimes \gamma^{*}\gamma=\gamma\otimes \alpha \alpha^{*}+q^{2}\gamma\otimes \gamma^{*}\gamma+(\alpha^{*}\otimes \gamma-\alpha^{*}\otimes \gamma)(I\otimes \alpha^{*})=$$
$$(\gamma\otimes \alpha +\alpha^{*}\otimes \gamma)(I\otimes \alpha^{*})+q^{2}\gamma\otimes \gamma^{*}\gamma-\alpha^{*}\otimes \gamma \alpha^{*}=$$
$$(\gamma\otimes \alpha +\alpha^{*}\otimes \gamma)(I\otimes \alpha^{*})-q(\alpha^{*}\otimes \alpha^{*}-q \gamma\otimes \gamma^{*})(I\otimes \gamma)=$$
$$\Delta(\gamma)(I\otimes \alpha^{*})-q\Delta(\alpha^{*})(I\otimes \gamma)$$
\end{proof}
\begin{rem}
When $q=1,$ we get $SU_{1}(2)\cong C(SU(2)).$ This is not hard to deduce as~\eqref{mm} shows that $SU_{1}(2)$ is commutative and that the spectrum of the algebra is homeomorphic to $S^{3}.$ Furthermore, it is also easy to see that this is an isomorphism of compact groups (recall that the spectrum of a commutative CQG has a compact group structure).
\end{rem}
\subsection{The $*$-Hopf algebra $SU_{q}^{0}(2)$}
As usual, inside every CQG lives a $*$-Hopf algebra and $SU_{q}(2)$ is no exception. We assume now $q\in (0,1]$
and define:
\begin{defn}\label{er}
Let $SU^{0}_{q}(2)\subseteq SU_{q}(2)$ be the $*$-subalgebra of $SU_{q}(2)$ generated by $\alpha$ and $\gamma$ (i.e no closure or similar topological properties). As $\Delta(SU_{q}^{0})\subseteq SU_{q}^{0}\otimes SU_{q}^{0}$ (notice that $SU_{q}^{0}\otimes SU_{q}^{0}\subseteq SU_{q}\otimes SU_{q}$ is isomorphic to the \textit{algebraic tensor product} of $SU_{q}^{0}$ with itself) we have an $*$-algebra with comultiplication. We now define a counit $\epsilon$ and antipode $S$ by the formulas
\begin{equation}\label{ee}
\begin{array}{ccc}
\epsilon(\alpha)=1, & \epsilon(\gamma)=0,
\end{array}
\end{equation}

\begin{equation}\label{bb}
\begin{array}{ccc}
S(\alpha)=\alpha^{*}, &  S(\alpha^{*})=\alpha, &  S(\gamma)=-q \gamma,\\
&  S(\gamma^{*})=-q^{-1}\gamma^{*}.
\end{array}
\end{equation}
We can now extend these formulas to all of $SU^{0}_{q}(2)$ by  require that $\epsilon$ is a homomorphism $SU^{0}_{q}(2)\rightarrow \mathbb{C}$ and $S$ to be an antihomomorphism $SU_{q}^{0}(2)\rightarrow SU_{q}^{0}(2).$
\end{defn}
Still there is the matter of whether or not these maps are actually well defined.
For this, we need a lemma. First we introduce some notation. For $k\in\mathbb{Z}$ and $n,m\in\mathbb{N},$ let
\begin{equation}
A_{k,n,m}=
\left\{\begin{array}{ccc}
 \alpha^{k}\gamma^{*n}\gamma^{m}, & \text{If $k\geq 0$}\\
\alpha^{*(-k)}\gamma^{*n}\gamma^{m}, & \text{If $k<0$}
\end{array}\right\}
\end{equation}
\begin{lem}\label{bar}
The set $\{A_{k,n,m}\}_{k,m,n}$ is a basis for $SU_{q}^{0}(2).$ 
\end{lem}
\begin{proof}
Some applications of~\eqref{nn} gives that $\{A_{k,n,m}\}_{k,m,n}$ spans $SU^{0}_{q}(2),$ so we need to check linear independence. We split this into two cases $q\in (0,1)$ and $q=1$
 To prove it when $q\in(0,1)$, we define a convenient representation on a Hilbert space. Let $\mathcal{H}$ be a Hilbert space with basis $\{e_{r,s}\}_{r\in\mathbb{N},s\in\mathbb{Z}}.$ Now define $\alpha',\gamma'\in B(\mathcal{H})$ by 
\begin{equation}
\begin{array}{ccc}
\alpha' e_{r,s}=\sqrt{1-q^{2r}}e_{r-1,s}, & \gamma' e_{r,s}=q^{r}e_{r,s+1}.
\end{array}
\end{equation}
We leave it as an exercise for the reader to show that actually $\alpha'$ and $\gamma'$ satisfies the equation in~\eqref{mm} and hence there exists a homomorphism $\phi:SU_{q}(2)\rightarrow B(\mathcal{H})$ such that $\phi(\alpha)=\alpha'$ and $\phi(\gamma)=\gamma'.$ 
One also easily checks the formula
\begin{equation}\label{vv}
\phi(A_{k,n,m})e_{r,s}=\left(\prod_{l=0}^{|k|-1}\sqrt{1-q^{2r+\sgn(k)2l}}\right)q^{r(n+m)}e_{r-k,s+m-n}.
\end{equation}
We can now prove linear independence of $\{A_{k,n,m}\}_{k,m,n}.$
Assume that we have a linear dependence of the form $\sum_{j,l,p}\lambda_{i,j,l}A_{k_{i},n_{j},m_{l}}=0$ of minimal length (i.e no $\lambda_{i,j,l}$ can be omitted without making the sum nonzero and every index $(k_{i},n_{j},m_{l})$ only appear once). Applying $\phi$ to this sum and using~\eqref{vv} we see directly that $k_{i}$ and $m_{l}-n_{j}$ must be constant.
But then again using~\eqref{vv} we see that for $r\geq |k_{i}|,$ we have 
\begin{equation}\label{jj}
\sum_{i,j,l} \lambda_{i,j,l}q^{r(n_{j}+m_{l})}=0.
\end{equation}
If $b=m-n$ and $c=m+n,$ then $m=\frac{c+b}{2}$ and $n=\frac{c-b}{2},$ so if $m_{l}-n_{j}$ is constant then the function $(k_{i},n_{j},m_{l})\mapsto m_{l}+n_{j}$ is injective. So we can rename the coefficients $\lambda_{m_{l}+n_{j}}:=\lambda_{k_{i},n_{j},m_{l}}$ and extend the function $\lambda_{k}$ to a function on all of $k \in\mathbb{N}$ by setting it to zero elsewhere. Now write~\eqref{jj} as
\begin{equation}
\begin{array}{ccc}
P(q^{r})=0, & \forall r\geq |k_{i}| \\
\text{where $P(x)=\sum_{k=0}^{\max (m_{l}+n_{j})} \lambda_{k}x^{k}$} .
\end{array}
\end{equation}
This clearly implies $P(x)=0$ and this proves the linear independence when $q\in (0,1).$
\\

When $q=1,$ we are in the case of $C(SU(2)).$ So $A_{k,m.n}$ can be considered a function on $S^{3}.$ If we parameterize these functions as $$A_{k,m,n}(\lambda,\gamma,\theta)=e^{i k \lambda}\sin(\gamma)^{k}\cos(\gamma)^{m+n}e^{i (n-m)\theta}$$
Using some Fourier analysis arguments, it is easy to see that these functions are linearly independent.
\end{proof}
\begin{prop}
$SU^{0}_{q}(2)$ with the counit $\epsilon$ and antipode $S$ as defined above is a $*$-Hopf algebra. 
\end{prop}
\begin{proof}
Let $W$ be the free unital $*$-algebra generated by two elements $a,c$ and let $I$ be the two-sided $*$-ideal in $W$ generated by the relations~\eqref{nn} with $a=\alpha$ and $c=\gamma.$ The canonical $*$-homomorphism $\phi:W\rightarrow SU^{0}_{q}(2),$ sending $a$ to $\alpha$ and $c$ to $\gamma$ has a kernel containing $I$ and hence we have a $*$-homomorphism $\bar{\phi}:W/I\rightarrow SU_{q}^{0}(2).$ For $k\in\mathbb{Z}$ and $m,n\in\mathbb{N},$ let
$$\Lambda_{k,m,n}=
\left\{
\begin{array}{ccc}
a^{k}c^{* m}c^{n}, & \text{If $k\geq 0$}\\
a^{*(-k)}c^{* m}c^{n}, & \text{If $k< 0$}
\end{array}
\right\}.
$$
The set $\{\Lambda_{k,m,n}\}_{k,m,n}$ spans $W/I$ as a vector space (by the same argument as for $SU^{0}_{q}(2)$) and $\bar{\phi}\left(\Lambda_{k,m,n}\right)=A_{k,m,n}.$ Therefore, $\bar{\phi}$ is actually an $*$-isomorphism of $*$-algebras.
If we define a $*$-homomorphism $e:W\rightarrow \mathbb{C}$ by $e(a)=1$ and $e(c)=0,$ then it is easy to see that $I\subseteq\ker e$ and hence we get a $*$-homomorphism $\epsilon:W/I\cong SU^{0}_{q}(2)\rightarrow\mathbb{C}$ that coincides with the formula for $\epsilon$ given in Definition~\ref{er}. We now show $(I\otimes \epsilon)\Delta=(\epsilon\otimes I)\Delta=I.$ This can be checked easily for $\alpha$ and $\gamma,$ and as both $(I\otimes \epsilon)\Delta$ and $(\epsilon\otimes I)\Delta$ are $*$-homomorphisms $SU^{0}_{q}(2)\rightarrow SU_{q}^{0}(2),$ this extends to all of $SU^{0}_{q}(2).$
\\

Define the homomorphism $\Sigma:W\rightarrow SU^{0}_{q}(2)^{\text{op}}$ (note, no $*$) by
\begin{equation}
\begin{array}{ccc}
\Sigma(a)=\alpha^{*}, & \Sigma(a^{*})=\alpha, & \Sigma(c)=-q \gamma, \\
& \Sigma(c^{*})=-q^{-1}\gamma^{*}.
\end{array}
\end{equation}
We have 
\begin{equation}
\begin{array}{ccc}
\Sigma(a^{*}a+c^{*}c-I)=\alpha^{*}\alpha+\gamma^{*}\gamma-I=0, & \Sigma(a a^{*}+q^{2}c^{*}c-I)=\alpha \alpha^{*}+q^{2}\gamma^{*}\gamma-I=0 \\
\Sigma(c c^{*}-c^{*}c)=\gamma \gamma^{*}-\gamma^{*}\gamma=0, & \Sigma(a c-q c a)=-q(\gamma \alpha^{*}-q \alpha^{*} \gamma)=0\\
 \Sigma( c^{*} a^{*}-q a^{*} c^{*})=-q^{-1}( \alpha \gamma^{*}-q \gamma^{*}\alpha)=0, &  \Sigma(a c^{*}-q c^{*} a)=-q^{-1}(\gamma^{*} \alpha^{*}-q^{*} \alpha^{*} \gamma)=0,\\
  \Sigma( c a^{*}-q a^{*} c)=-q( \alpha^{*} \gamma^{*}-q \gamma^{*}\alpha^{*})=0.
\end{array}
\end{equation}
This shows that $I$ is in the kernel of $\Sigma.$ As a $*$-ideal, $I$ is trivially also an ideal and hence we can define a homomorphism
$S:W/I\cong SU_{q}^{0}(2)\rightarrow SU_{q}^{0}(2)^{\text{op}}.$ Again it is easy to see that our new map $S$ coincide with the one in Definition~\ref{er}. The formula $S(S(b^{*})^{*})=b,\forall b\in SU_{q}^{0}(2)$ follows from that the map $b\mapsto S(S(b^{*}))^{*}$ is a composition of two linear anti-homomorphisms and two anti-linear anti-homomorphisms, and hence is a homomorphism $ SU_{q}^{0}(2)\rightarrow SU_{q}^{0}(2).$ It can be easily checked that this homomorphism is the identity on $\alpha,\alpha^{*},\gamma,\gamma^{*}$ and hence must be the identity on all of $ SU_{q}^{0}(2).$ What is left to prove are the equations (let $m$ denote the multiplication)
\begin{equation}\label{st}
\begin{array}{ccc}
m(I\otimes S)\Delta(b)=\epsilon(b)I, & m(S\otimes I)\Delta(b)=\epsilon(b)I, & \forall b\in SU_{q}^{0}(2).
\end{array}
\end{equation}
Let $Q$ denote the set of all $b\in SU^{0}_{q}(2)$ such that~\eqref{st} is true, this is obviously a linear subspace. We will prove that $Q$ is an algebra and $\alpha,\alpha^{*},\gamma,\gamma^{*}\in Q,$ this will give us $Q=SU_{q}^{0}(2).$
\\

Step $1$: $\alpha,\alpha^{*},\gamma,\gamma^{*}\in Q.$ This is an easy calculation using the equations in~\eqref{nn}
\begin{equation}
\begin{array}{ccc}
m(S\otimes I)\Delta(\alpha)=m(S\otimes I)(\alpha\otimes \alpha -q \gamma^{*}\otimes \gamma)=\alpha^{*}\alpha+\gamma^{*}\gamma=I=\epsilon(\alpha)I\\
m(I\otimes S)\Delta(\alpha)=m(I\otimes S)(\alpha\otimes \alpha -q \gamma^{*}\otimes \gamma)=\alpha \alpha^{*}+q^{2}\gamma^{*}\gamma=I=\epsilon(\alpha)I\\
m(S\otimes I)\Delta(\alpha^{*})=m(S\otimes I)(\alpha^{*}\otimes \alpha^{*} -q \gamma \otimes \gamma^{*})=\alpha \alpha^{*}+q^{2}\gamma^{*}\gamma=I=\epsilon(\alpha^{*})I\\
m(I\otimes S)\Delta(\alpha^{*})=m(I\otimes S)(\alpha^{*}\otimes \alpha^{*} -q \gamma\otimes \gamma^{*})=\alpha^{*} \alpha+\gamma^{*}\gamma=I=\epsilon(\alpha^{*})I\\
\\
m(S\otimes I)\Delta(\gamma)=m(S\otimes I)(\gamma\otimes \alpha + \alpha^{*}\otimes \gamma)=-q \gamma\alpha +\alpha \gamma=0=\epsilon(\gamma)I\\
m(I\otimes S)\Delta(\gamma)=m(I \otimes S)(\gamma\otimes \alpha + \alpha^{*}\otimes \gamma)=\gamma \alpha^{*}- q \alpha^{*}\gamma=0=\epsilon(\gamma)I\\
m(S\otimes I)\Delta(\gamma^{*})=m(S\otimes I)(\gamma^{*}\otimes \alpha^{*} + \alpha\otimes \gamma^{*})=-q^{-1} \gamma^{*} \alpha^{*}+\alpha^{*} \gamma^{*}=0=\epsilon(\gamma^{*})I\\
m(I\otimes S)\Delta(\gamma^{*})=m(I \otimes S)(\gamma^{*}\otimes \alpha^{*} + \alpha\otimes \gamma^{*})=\gamma^{*} \alpha-  q^{-1}\alpha^{*}\gamma^{*}=0=\epsilon(\gamma^{*})I
\end{array}
\end{equation}
\\

Step $2$: $Q$ is an algebra. Assume $b,d\in Q.$ Using the Sweedler notation, we calculate
$$m(I\otimes S)\Delta(bd)=m(I\otimes S)\sum (bd)_{1}\otimes (bd)_{2}=$$
$$\sum b_{1}d_{1} S(d_{2})S(b_{2})=\epsilon(d)\sum b_{1}S(b_{2})=\epsilon(d)\epsilon(b) I=\epsilon(bd)I.$$
A similar calculation gives $m(S\otimes I)\Delta(bd)=\epsilon(bd)I.$ The proof is now complete!
\end{proof}
\begin{rem}
Notice that when $q\in(0,1),$ then $SU_{q}^{0}(2)$ is neither commutative or cocommutative, nor is $S$ commuting with the $*$ operation.
\end{rem}
\section{Finite dimensional representations of $SU_{q}(2)$.}
Recall that a finite dimensional representation (FDR) of a CQG $G$ on a finite dimensional Hilbert space $\mathcal{H}$ is an invertible $U\in B(\mathcal{H})\otimes C(G)$ such that $\iota\otimes \Delta (U)=U_{1,2}U_{1,3}.$
\\

If $\phi:C(G)\rightarrow C(F)$ is a morphism of CQG's, then if $U\in B(\mathcal{H})\otimes C(G)$ is a FDR, we get a new FDR $(\iota\otimes \phi)(U)\in B(\mathcal{H})\otimes C(F).$
Furthermore, this process is also compatible with direct sums and tensor product (let here $\boxplus$ denote direct sum and $\boxtimes$ denote the tensor product of CQG representations) so that $(\iota\otimes\phi)(U\boxtimes W)=(\iota\otimes \phi)(U)\boxtimes (\iota\otimes \phi)(W).$ Also, if $W$ is unitary, then so is $(\iota\otimes \phi)(W).$ These statements are all easily checked.
\\

This simple observation is one of the key ideas when classifying the FDR's for $SU_{q}(2),$  and it begins with the following results:
\begin{prop}\label{an}
Let $z\in C(\mathbb{T})$ be the identity function (on $\mathbb{T}$). There exists a unique morphism of CQG's $\pi: SU_{q}(2)\rightarrow C(\mathbb{T})$ such that $\pi (\alpha)=z$ and $\pi(\gamma)=0.$ 
\\

Moreover, for each $\zeta\in \mathbb{T},$ there exists a unique $*$-homomorphism $\theta_{\zeta}:SU_{q}(2)\rightarrow \mathbb{C}$ such that $\theta_{\zeta}(\alpha)=\zeta$ and $\theta_{\zeta}(\gamma)=0.$ Also $\forall \zeta,\eta\in\mathbb{T},$ we have
$\theta_{\zeta}*\theta_{\eta}=\theta_{\zeta \eta}.$
\end{prop}
\begin{proof}
Consider $C(\mathbb{T})$ as a $C^{*}$-algebra in $B(L^{2}(\mathbb{T})).$ Putting $\alpha=z$ and $\gamma=0$ in~\eqref{mm}, we get a unitary matrix for all $q\in[0,1]$ and hence there is, by the definition of $SU_{q}(2),$ a homomorphism $\pi: SU_{q}(2)\rightarrow C(\mathbb{T}).$ It is obvious that $\pi$ is the unique homomorphism with the properties mention in the statement of the proposition We need now to check that $\pi$ is a homomorphism of CQG's. Using the formulas for $\Delta(\alpha)$ and $\Delta(\gamma),$ it is easy to see that $(\pi\otimes \pi )\Delta(\alpha)=z\otimes z=\Delta_{\mathbb{T}}(\pi(\alpha))$ and $(\pi\otimes\pi)\Delta(\gamma)=0=\Delta_{\mathbb{T}}(\pi(\gamma)).$ The rest follows from that $\alpha$ and $\gamma$ generates $SU_{q}(2).$
\\

For the second part, consider the evaluation at $\zeta\in\mathbb{T},$ $\bar{\theta}_{\zeta}:C(\mathbb{T})\rightarrow\mathbb{C}.$ It is easy to see that $\bar{\theta}_{\zeta}*\bar{\theta}_{\eta}=\bar{\theta}_{\zeta\eta}$ (the convolution is now made in $C(\mathbb{T})$). We can now define $\theta_{\zeta}(.)=\bar{\theta}_{\zeta}(\pi(.)).$ Again, uniqueness is easy to see.
As $\pi$ is a morphism of CQG's, we get $$\theta_{\zeta}*\theta_{\eta}(.)=(\bar{\theta}_{\zeta}*\bar{\theta}_{\eta})(\pi(.))=\bar{\theta}_{\zeta\eta}(\pi(.))=\theta_{\zeta\eta}(.).$$
\end{proof}
As we know, every FDR $W\in B(\mathcal{H})\otimes C(\mathbb{T})$ is equivalent to a group representation $\mu:\mathbb{T}\rightarrow B(\mathcal{H}).$ Moreover, irreducible representation of $\mathbb{T}$ is of the form $\zeta\mapsto \zeta^{k}$ for $k\in\mathbb{Z}.$ Thus every such representation $\mu$ decomposes $\mathcal{H}$ into a direct sum
$\mathcal{H}=\oplus \mathcal{H}_{k}$ such that the restriction $\mu |\mathcal{H}_{k}$ is of the form $\zeta\mapsto \zeta^{k}I.$
\\

For every  FDR $W\in B(\mathcal{H})\otimes C(\mathbb{T})$ we now define a function $m_{W}(k):\mathbb{Z}\rightarrow \mathbb{N},$ by $m_{W}(k)=\dim \mathcal{H}_{k}.$ The following formulas are left to the reader to prove :)
$$
\begin{array}{ccc}
m_{W\boxplus U}(k)=m_{W}(k)+m_{U}(k), & m_{W\boxtimes U}(k)=\sum_{m+n=k} m_{W}(m)m_{U}(n), & \forall k\in\mathbb{Z}.
\end{array}
$$ 
If we know use the morphism $\pi$ from Proposition~\ref{an}, we can map FDR of $SU_{q}(2)$ to FDR of $C(\mathbb{T})$ as in the discussion at the start of the this section. Further properties of this mapping (that was also discussed at the start) allows us to deduce this lemma:
\begin{lem}
Let $W$ be an FDR of $SU_{q}(2)$, let $M_{W}$ be the function $\mathbb{Z}\rightarrow\mathbb{N}$ defined by $M_{W}(k)=m_{(\iota\otimes \pi)(W)}(k),$ then
\begin{equation}\label{ll}
\begin{array}{ccc}
M_{W\boxplus U}(k)=M_{W}(k)+M_{U}(k), & M_{W\boxtimes U}(k)=\sum_{m+n=k} M_{W}(m)M_{U}(n), & \forall k\in\mathbb{Z}.
\end{array}
\end{equation}
We call $M_{W}$ the weight function of $W.$
\end{lem}
We can now formulate the main theorem.
\begin{thm}[Woronowicz]\label{woroz}
Let $q\in(0,1).$ Up to equivalence, there is for each $n\in\mathbb{N}$ just one irreducible representation $U_{n}$ of dimension $n+1.$ The weight function $M_{(n)}:=M_{U_{n}}$ of $U_{n}$ is given by the formula
\begin{equation}\label{wf}
M_{(n)}(k)=\left\{
\begin{array}{ccc}
1, & k\in\{-n,-n+2,...,n-2,n\}\\
0, & \text{otherwise}
\end{array}\right\}.
\end{equation}
Furthermore, a unitary FDR $W$ is up to equivalence determined by its weight function.

\end{thm}
\begin{rem}
This is of course true for $q=1$ also, but as this is well known and since we want to avoid devoting a part of every proof to prove it in the case $q=1,$ we have omitted the proof of this case.
\end{rem}
\begin{subsection}{Proof of Theorem~\ref{woroz}}
Be begin first by finding candidates for the the sequence of representations mentioned in the statement of Theorem~\ref{woroz}.
\begin{defn}
For $n\in\mathbb{N},$ define $V_{n}\subseteq SU^{0}_{q}(2)$ as the linear subspace of $SU^{0}_{q}(2)$ spanned by the basis vectors $\alpha^{k}\gamma^{*(n-k)}$ for $0\leq k\leq n.$
\end{defn}

\begin{lem}
The vector space $V_{n}$ is a right ideal in $\tilde{SU_{q}(2)}$ i.e $\Delta(V_{n})\subseteq V_{n}\otimes \tilde{SU_{q}(2)}$
\end{lem}
\begin{proof}
By multiplicity of the coproduct, we have $$\Delta(\alpha^{k}\gamma^{*(n-k)})=\Delta(\alpha)^{k}\Delta(\gamma^{*})^{n-k}=$$
$$(\alpha\otimes\alpha-q \gamma^{*}\otimes \gamma)^{k}(\gamma^{*}\otimes \alpha+\alpha\otimes \gamma^{*})^{n-k}.$$
It is not hard to see that the left tensor factors in the expansion of this sum can, by the use of the commutations relations in $\tilde{SU}_{q}(2)$ be written as $\alpha^{j}\gamma^{*(n-j)},$ for $0 \leq j\leq n.$ 
\end{proof}
We now show how to define a corepresentation on $V_{n}.$ From the above lemma, we know that $\Delta(\alpha^{k}\gamma^{*(n-k)})=\sum_{j=0}^{n}\alpha^{j}\gamma^{*(n-j)}\otimes u_{j,k}$ for elements $u_{j,k}\in SU^{0}_{q}(2).$
\begin{defn}\label{un}
For each $V_{n},$ we define $U_{n}:=(u_{j,k})_{j,k}\in M(V_{n})\otimes SU^{0}_{q}(2).$
\end{defn}

\begin{prop}
The matrix $U_{n}$ is a representation of $SU_{q}(2)$ for every $n\in\mathbb{N}.$ Furthermore, the weight function of $U_{n}$ is the function $M_{(n)}(k)$ from Theorem~\ref{woroz}.
\end{prop}
\begin{proof}
We start by checking the property $I\otimes\Delta(U_{n})=(U_{n})_{1,2}(U_{n})_{1,3}.$ This is equivalent to checking 
\begin{equation}\label{eq2}
\Delta(u_{j,k})=\sum_{m=0}^{n}u_{j,m}\otimes u_{m,k}.
\end{equation}
Consider the sum 
\begin{equation}\label{eq1}
\sum_{j=0}^{n}\alpha^{j}\gamma^{*(n-j)}\otimes \Delta( u_{j,k}).
\end{equation}
By the definition of $u_{j,k},$ the sum~\eqref{eq1} is the same as $(I\otimes \Delta)\Delta(\alpha^{k}\gamma^{*(n-k)}).$ By coassociativity, we now have
$$(I\otimes \Delta)\Delta(\alpha^{k}\gamma^{*(n-k)})=( \Delta \otimes I)\Delta(\alpha^{k}\gamma^{*(n-k)})=\sum_{j,m=0}^{n}\alpha^{m}\gamma^{*(n-m)} \otimes u_{m,j}\otimes u_{j,k} .$$
Comparison of the basis vectors in the foremost left tensor factor now gives~\eqref{eq2}.
\\

We are going to show invertibility of $U_{n}$ by finding an inverse. Lets look at the matrix $(S(u_{j,k}))_{j,k}=I\otimes S(U_{n}).$ By above
$$\sum_{j,m=0}^{n}\alpha^{m}\gamma^{*(n-m)} \otimes S(u_{m,j}) u_{j,k}=(I\otimes m)(I\otimes S\otimes I)(I\otimes \Delta)\Delta(\alpha^{k}\gamma^{*(n-k)}),$$
where $m$ is the multiplication operator $SU_{q}^{0}(2)\otimes SU_{q}^{0}(2)\rightarrow SU_{q}^{0}(2).$
A property of the antipode is that $m(I\otimes S)\Delta=m(S\otimes I)\Delta=\epsilon.$ Hence 
$$(I\otimes m)(I\otimes S\otimes I)(I\otimes \Delta)\Delta(\alpha^{k}\gamma^{*(n-k)})=(I\otimes \epsilon)\Delta(\alpha^{k}\gamma^{*(n-k)})=\alpha^{k}\gamma^{*(n-k)}.$$
This gives $\sum_{j,m=0}^{n}\alpha^{m}\gamma^{*(n-m)} \otimes S(u_{m,j}) u_{j,k}=\sum_{j,m=0}^{n}\alpha^{m}\gamma^{*(n-m)} \otimes \delta_{m,k}$ and so $(S(u_{j,k}))_{j,k}$ is a left inverse for $U_{n}.$ Similar arguments shows that it is also a right inverse.
\\

To prove the claims about the weight function, recall the homomorphism $\pi:SU_{q}(2)\rightarrow C(\mathbb{T})$ in Proposition~\ref{an}.
The $*$-homomorphism $(I\otimes \pi)\Delta :SU_{q}\rightarrow SU_{q}(2)\otimes C(\mathbb{T}) $ maps $\alpha\mapsto \alpha \otimes z$ and $\gamma \mapsto \gamma\otimes z.$ Hence $(I\otimes \pi)\Delta(\alpha^{k}\gamma^{*(n-k)})=\alpha^{k}\gamma^{*(n-k)}\otimes z^{n-2k}.$ But we have also 
$$(I\otimes \pi)\Delta(\alpha^{k}\gamma^{*(n-k)})=\sum_{j=0}^{n}\alpha^{j}\gamma^{*(n-j)}\otimes \pi(u_{j,k}).$$
In conclusion, the matrix $(I\otimes \pi)(U_{n})$ is diagonal with diagonal elements 

$\{z^{-n},z^{2-n},...,z^{n-2},z^{n}\}.$ This shows that $M_{U_{n}}$ must have the claimed form.
\end{proof}
So we now have constructed some good candidates for the irreducible representations of $SU_{q}(2).$ What is left to prove is that $U_{n}$ is actually irreducible and that any other irreducible representation $U$ of the same dimension is equivalent to $U_{n}$. This is the tricky part.

\begin{defn}
Let $R$ be an algebra over $\mathbb{C}.$ A $SU_{q}^{0}(2)$-point of $R$ is a 4-tuple 

$a,a^{(*)},c,c^{(*)}\in R$ such that 
\begin{equation}\label{nn}
\begin{array}{ccc}
a^{(*)} a+c^{(*)} c=I, & a a^{(*)}+ q^{2}c^{(*)} c=I, & c^{(*)}c=c c^{(*)},\\
a c=q c a, & a c^{(*)}=q c^{(*)} a.\\
c^{(*)}a^{(*)} =q a^{(*)}c^{(*)}, &  c a^{(*)}=q a^{(*)} c. 
\end{array}
\end{equation}
(note that we do not necessarily have an involution on $R,$ the $(*)$ is just a symbol).
\end{defn}

\begin{lem}
If $R$ is an algebra over $\mathbb{C},$ then set of $SU_{q}^{0}(2)$-points of $R$ is isomorphic to $Hom(SU_{q}^{0}(2),R)$ (the set of \textbf{algebra} morphism $SU_{q}^{0}(2)\rightarrow R$) 
\end{lem}
\begin{proof}
This follows easily from Lemma~\ref{bar}.
\end{proof}
We now a define a $SU_{q}^{0}(2)$-point in $M_{4}(\mathbb{C})$ that will be essential in the constructions to come. Let
$$a_{m}=\left[ 
\begin{array}{cccc}
1 & 0 & 1 & 0\\
0 & q^{-1} & 0 & 0\\
0 & 0 & q^{-2} & 0\\
0 & 0 & 0 & q^{-1}\\

\end{array} 
\right],a_{m}^{(*)}=\left[ 
\begin{array}{cccc}
1 & 0 &- q^{2} & 0\\
0 & q & 0 & 0\\
0 & 0 & q^{2} & 0\\
0 & 0 & 0 & q\\

\end{array} 
\right]$$
$$
,c_{m}=\left[ 
\begin{array}{cccc}
0 & 0 & 0 & -q\\
0 & 0 & 0 & 0\\
0 & 0 & 0 & 0\\
0 & 0 & 0 & 0\\

\end{array} 
\right]
,c_{m}^{(*)}=\left[ 
\begin{array}{cccc}
0 & -q^{-1} & 0 & 0\\
0 & 0 & 0 & 0\\
0 & 0 & 0 & 0\\
0 & 0 & 0 & 0\\

\end{array} 
\right]
.
$$
\begin{lem}
$a_{m},a_{m}^{(*)},c_{m},c_{m}^{(*)}\in M_{4}(\mathbb{C})$ is a $SU_{q}^{0}$ point of $M_{4}(\mathbb{C}).$ 
\end{lem}
\begin{proof}
Note that $a_{m}a_{m}^{(*)}=a_{m}^{(*)}a_{m}=I$ and $c_{m}c_{m}^{(*)}=c_{m}^{(*)}c_{m}=0,$ so the first row of equations in~\eqref{nn} is true.
The rest is an easy verification and is left to the reader :).
\end{proof}

So we have an algebra homomorphism $\phi:SU_{q}^{0}\rightarrow M_{4}(\mathbb{C}).$ Since $a_{m},a_{m}^{(*)},c_{m},c_{m}^{(*)}$ are all upper triangular, the range of $\phi$ consists of upper triangular matrices. More is true though, since the matrices $a_{m},a_{m}^{(*)},c_{m},c_{m}^{(*)}$ only has nonzero entries in the diagonal and the first row and it is not hard to see that such matrices is also a subalgebra of $M_{4}(\mathbb{C}).$ Thus for every $b\in SU_{q}^{0}(2),$ we have 
\begin{equation}\label{qeqe}
\phi(b)=\left[ 
\begin{array}{cccc}
\epsilon(b) & \chi_{0}(b) & \chi_{1}(b) & \chi_{2}(b)\\
0 & f_{0}(b) & 0 & 0\\
0 & 0 & f_{1}(b) & 0\\
0 & 0 & 0 & f_{2}(b)\\

\end{array} \right]
\end{equation}
for some linear functionals $\epsilon,f_{0},f_{1}, f_{2}, \chi_{0} ,\chi_{1}, \chi_{2}$ (we will soon prove that the $\epsilon$ in~\eqref{qeqe} is the same as the counit on $SU_{q}^{0}(2)$, hence the use of the same symbol).
\begin{prop}[Basic properties]
The linear functionals $\epsilon,f_{0},f_{1}, f_{2}, \chi_{0} ,\chi_{1}, \chi_{2}$ have the following properties.
\begin{enumerate}
\item  $f_{0},f_{1},f_{2}$ are algebra morphisms $SU_{q}^{0}(2)\rightarrow \mathbb{C}$ and $\epsilon$ is the same as the coproduct in $SU_{q}^{0}(2).$ Also $f_{0}=f_{2}.$
\item For each $\chi_{k}$ and $b,d\in SU_{q}^{0}(2),$ we have 
\begin{equation}\label{xi}
\chi_{k}(b d)=\epsilon(b)\chi_{k}(d)+\chi_{k}(b)f_{k}(d).
\end{equation}
\item We have 
\begin{equation}\label{lala}
\chi_{1}=\frac{q^{2}}{1-q^{2}}(f_{1}-\epsilon)
\end{equation}
\end{enumerate}
\end{prop}
\begin{proof}
$(1)$ As the range of the homomorphism $\phi$ consists of upper triangular matrices, it follows that projection onto the diagonal elements must be homomorphisms to $\mathbb{C}.$ It is easy to see that $\epsilon$ is the same as the counit on the generators and hence on all of $SU_{q}^{0}(2).$ Similary, $f_{0}=f_{2}$ on the generators and so also equal on all of $SU_{q}^{0}(2).$
\\

$(2)$ This follows from a comparison of the first row elements in the equality
$$
\left[ 
\begin{array}{cccc}
\epsilon(b d) & \chi_{0}(b d) & \chi_{1}(b d) & \chi_{2}(b d)\\
0 & f_{0}(b d) & 0 & 0\\
0 & 0 & f_{1}(b d) & 0\\
0 & 0 & 0 & f_{2}(b d)\\

\end{array} \right]=
$$
$$
\left[ 
\begin{array}{cccc}
\epsilon(b) & \chi_{0}(b) & \chi_{1}(b) & \chi_{2}(b)\\
0 & f_{0}(b) & 0 & 0\\
0 & 0 & f_{1}(b) & 0\\
0 & 0 & 0 & f_{2}(b)\\

\end{array} \right].\left[ 
\begin{array}{cccc}
\epsilon(d) & \chi_{0}(d) & \chi_{1}(d) & \chi_{2}(d)\\
0 & f_{0}(d) & 0 & 0\\
0 & 0 & f_{1}(d) & 0\\
0 & 0 & 0 & f_{2}(d)\\

\end{array} \right].
$$
\\

$(3)$ An quick calculation showes that the matrices $a_{m},a_{m}^{(*)},c_{m},c_{m}^{(*)}$ all commutes with the matrix
$$\left[ 
\begin{array}{cccc}
0 & 0 & 1 & 0\\
0 & 0 & 0 & 0\\
0 & 0 & \frac{1-q^{2}}{q^{2}} & 0\\
0 & 0 & 0 & 0\\

\end{array} \right].$$ Thus all of the range of $\phi$ commutes with this matrix. $(3)$ follows now from comparing the elements in the equality
$$
\left[ 
\begin{array}{cccc}
\epsilon(b) & \chi_{0}(b) & \chi_{1}(b) & \chi_{2}(b)\\
0 & f_{0}(b) & 0 & 0\\
0 & 0 & f_{1}(b) & 0\\
0 & 0 & 0 & f_{2}(b)\\

\end{array} \right].\left[ 
\begin{array}{cccc}
0 & 0 & 1 & 0\\
0 & 0 & 0 & 0\\
0 & 0 & \frac{1-q^{2}}{q^{2}} & 0\\
0 & 0 & 0 & 0\\

\end{array} \right]=
$$
$$
\left[ 
\begin{array}{cccc}
0 & 0 & 1 & 0\\
0 & 0 & 0 & 0\\
0 & 0 & \frac{1-q^{2}}{q^{2}} & 0\\
0 & 0 & 0 & 0\\

\end{array} \right].\left[ 
\begin{array}{cccc}
\epsilon(b) & \chi_{0}(b) & \chi_{1}(b) & \chi_{2}(b)\\
0 & f_{0}(b) & 0 & 0\\
0 & 0 & f_{1}(b) & 0\\
0 & 0 & 0 & f_{2}(b)\\

\end{array} \right]
$$
\end{proof}
We now give some helpful tables relating to the functionals $f_{0},f_{1}, f_{2}, \chi_{0} ,\chi_{1}, \chi_{2}.$
\begin{equation}
\begin{tabular}{|r c l|}
  \hline
& & \\
  $f_{0}(\alpha)=q^{-1}$ & $f_{1}(\alpha)=q^{-2}$ &  $f_{0}(\alpha)=q^{-1}$\\
   $f_{0}(\gamma)=0$ & $f_{1}(\gamma)=0$ &  $f_{0}(\gamma) =0$\\ 
 $f_{0}(\alpha^{*})=q^{1}$ & $f_{1}(\alpha^{*})=q^{2}$ &  $f_{0}(\alpha^{*})=q^{1}$ \\
   $f_{0}(\gamma^{*})=$0 & $f_{1}(\gamma^{*})=0$ & $ f_{0}(\gamma^{*}) =0$\\
& & \\
  \hline
\end{tabular}
\end{equation}
\begin{equation}
\begin{tabular}{|r c l|}
  \hline
& Non-zero values of $\chi_{k}$ & \\
  $\chi_{0}(\gamma^{*})=-q^{-1}$ & $\chi_{1}(\alpha)=1 $ &  $\chi_{1}(\alpha^{*})=-q^{2}$\\
   $\chi_{2}(\gamma)=-q $ & & \\ 

& All other values of $\chi_{k}$ on the generators are zero & \\
  \hline
\end{tabular}
\end{equation}

Given $g,h\in (SU_{q}^{2}(2))^{*}$ and $b\in SU_{q}^{0}(2),$ recall that the convolution $g*h \in  (SU_{q}^{2}(2))^{*}$ between $g$ and $h$ was defined as $(g\otimes h) \Delta$, the left convolution $g*b\in SU_{q}^{0}(2)$ between $g$ and $b$ was defined as $(I\otimes g) \Delta(b)$ and the right convolution $b*g$ was $(g\otimes I)\Delta(b).$ The convolution is associative and makes $(SU_{q}^{2}(2))^{*}$ into an algebra with the product $*$.

\begin{equation}\label{chicon}
\begin{tabular}{|r c l|}
  \hline
& Convolutions with $\chi_{k}$ & \\
$\chi_{0}*\alpha=0$ & $\chi_{1}*\alpha=\alpha $ &  $\chi_{1}*\alpha=q^{2}\gamma^{*}$\\
$\chi_{0}*\gamma=0$ & $\chi_{1}*\gamma=\gamma $ &  $\chi_{1}*\gamma=-q\alpha^{*}$\\
$\chi_{0}*\alpha^{*}=\gamma^{*} $ & $\chi_{1}*\alpha^{*}=-q^{2}\alpha^{*} $ &  $\chi_{1}*\alpha^{*}=0$\\
$\chi_{0}*\gamma^{*}=-q^{-1}\alpha$ & $\chi_{1}*\gamma^{*}=-q^{2}\gamma^{*} $ &  $\chi_{1}*\gamma^{*}=0$\\

&  & \\
  \hline
\end{tabular}
\end{equation}
\begin{equation}
\begin{tabular}{|r c l|}
  \hline
& Convolutions with $f_{k}$ & \\
$f_{0}*\alpha=q^{-1}\alpha $ & $f_{1}*\alpha=q^{-2}\alpha $ &  $f_{2}$\\
$f_{0}*\gamma=q^{-1}\gamma$ & $f_{1}*\gamma=q^{-2}\gamma $ &  SAME\\
$f_{0}*\alpha^{*}=q\alpha^{*}$ & $f_{1}*\alpha^{*}=q^{2}\alpha^{*} $ &  AS\\
$f_{0}*\gamma^{*}=q\gamma^{*}$ & $f_{1}*\gamma^{*}=q^{2}\gamma^{*} $ &  $f_{0}$\\

&  & \\
  \hline
\end{tabular}
\end{equation}

 Recall also the we had an $*$-involution on $(SU_{q}^{2}(2))^{*}$ that was defined for $g\in(SU_{q}^{2}(2))^{*}$ as $g^{*}(b)=\overline{g((S(b))^{*})}.$ As $(S\otimes S)\Delta=\Delta^{op} S,$ we have for $g,h \in (SU_{q}^{2}(2))^{*}$ and $b\in SU_{q}^{0}(2)$ that 
$$(g^{*}*h^{*})(b)=\overline{(g\otimes h)((S\otimes S)\Delta(b))^{*}}=(g\otimes h)\Delta^{op}((S(b))^{*})=$$
$$(h\otimes g)\Delta((S(b))^{*})=(h*g)^{*}(b).$$
So $(h*g)^{*}=g^{*}*h^{*}$ and since also $(\lambda g)^{*}=\overline{\lambda}g^{*}$ for $\lambda\in\mathbb{C}$ and $(g+h)^{*}=g^{*}+h^{*},$  we have a proper $*$-involution on $(SU_{q}^{2}(2))^{*}$ with the convolution product.
\\

We will now prove some convolution identities of the functionals $f_{0},f_{1}, f_{2}, \chi_{0} ,\chi_{1}, \chi_{2}$ that will make up the technical heart of this proof.

\begin{prop}[Convolution properties]\label{conprop}
The following equalities holds:
\begin{enumerate}
\item $f_{0}*f_{0}=f_{1}$
\item 

\begin{enumerate}
\item
$-q\chi_{0}^{*}=\chi_{2}$
\item
$\chi_{1}^{*}=\chi_{1}$  

\end{enumerate}
\item
\begin{enumerate}

\item
$\chi_{0}*f_{0}=q^{2}f_{0}*\chi_{0}$

\item
 $q^{2}\chi_{2}*f_{0}=f_{0}*\chi_{2}$
\item
$\chi_{0}*f_{1}=q^{4}f_{1}*\chi_{0}$
\item
 $q^{4}\chi_{2}*f_{1}=f_{1}*\chi_{2}$
\item
$f_{0}*\chi_{1}=\chi_{1}*f_{0}$
\item
$f_{1}*\chi_{1}=\chi_{1}*f_{1}$
\end{enumerate}
\item
\begin{enumerate}

\item
$q \chi_{2}*\chi_{0}-q^{-1}\chi_{0}*\chi_{2}=\chi_{1}$
\item
$q^{2}\chi_{1}*\chi_{0}-q^{-2}\chi_{0}*\chi_{1}=(1+q^{2})\chi_{0}$
\item
$q^{2}\chi_{2}*\chi_{1}-q^{-2}\chi_{1}*\chi_{2}=(1+q^{2})\chi_{2}$ 
\end{enumerate}

\end{enumerate}
\end{prop}
\begin{rem}
Compare (2) with~\eqref{su1} from the introduction.
\end{rem}
\begin{proof}
For easy reference we list the action of coproduct and antipode on the generators again
$$
\begin{array}{cc}
\Delta(\alpha)=\alpha\otimes \alpha-q \gamma^{*}\otimes \gamma , &  \Delta (\gamma)= \gamma\otimes \alpha+ \alpha^{*}\otimes \gamma\\

\end{array}
$$
$$
\begin{array}{ccc}
S(\alpha)=\alpha^{*}, &  S(\alpha^{*})=\alpha, &  S(\gamma)=-q \gamma,\\
&  S(\gamma^{*})=-q^{-1}\gamma^{*}.
\end{array}
$$
$(1)$ is easy since $f_{0}*f_{0}$ is still an algebra homomorphism to $\mathbb{C}$ and a quick calculation gives that its values on the generators coincide with $f_{1}$'s.
\\

If we there are $\chi,\chi',f,f'\in(SU_{q}^{2}(2))^{*}$ such that $\chi(bd)=f(b)\chi(d)+\chi(b)f'(d)$ and $\chi'(bd)=f(b)\chi'(d)+\chi'(b)f'(d)$ for all $b,d\in SU_{q}^{0}(2),$ then to show that $\chi=\chi',$ we only need to check this equality on the generators of $SU_{q}^{0}(2).$  We have for all $b,d\in SU_{q}^{0}(2)$ that
$$
\chi_{k}*(bd)=(I\otimes \chi_{k})\Delta(bd)=(I\otimes\chi_{k})\sum b_{(1)}d_{(1)}\otimes b_{(2)}d_{(2)}=
\sum b_{(1)}d_{(1)}\otimes \chi_{k}(b_{(2)}d_{(2)})=
$$
$$
\sum b_{(1)}d_{(1)}\otimes\epsilon( b_{(2)})\chi_{k}(d_{(2)})+\sum b_{(1)}d_{(1)}\otimes \chi_{k}( b_{(2)})f_{k}(d_{(2)})=
$$
$$
\left(\sum b_{(1)}\otimes \epsilon(b_{(2)})\right)\left(\sum d_{(1)}\otimes \chi_{k}(d_{(2)})\right)+\left(\sum b_{(1)}\otimes \chi_{}k(b_{(2)})\right)\left(\sum d_{(1)}\otimes f_{k}(d_{(2)})\right)=
$$

$$
b(\chi_{k}*d)+(\chi_{k}*b)(f_{k}*d)
$$
and similarly $(bd)*\chi_{k}=b (d*\chi_{k})+(b*\chi_{k})(d*f_{k}) .$
\\

(2) Lets denote the left side of $(a)$ by $X.$ It is not hard to see that $X(bd)=\epsilon(b)X(d)+X(b)f_{0}(d)$ for all $b,d\in SU_{q}^{0}(2)$ and hence by the above comment, we only need to check the equality on the set $\alpha,\alpha^{*},\gamma,\gamma^{*}$ and a calculation shows that the only non-zero value is $X(\gamma)=-q.$ Thus $X=\chi_{2}.$ (b) follows much in the same vein and is left to the reader.

(3) We will prove (a), then $(b)$ follows by taking the involution on $(a),$ using $2(a)$ and that $f^{*}_{0}=f_{0}$ (left for the reader to prove). We then get $(c)$ and $(d)$ by using $1,$ since, for example, $f_{1}*\chi_{0}=f_{0}*f_{0}*\chi_{0}=q^{-2}f_{0}*\chi*f_{0}=q^{-4}\chi_{0}*f_{0}*f_{0}=q^{-4}\chi_{0}*f_{1}.$ The two that remains $(e),(f)$ follows from the first four once we proven $4(a).$
\\

 As $f_{0}$ is a homomorphism and $(f_{0}*\chi_{0})(b)=f_{0}(\chi_{0}*b)$ and $(\chi_{0}*f_{0})(b)=f_{0}(b*\chi_{0})$, we get
$$
(f_{0}*\chi_{0})(bd)=f_{0}(b)(f_{0}*\chi_{0})(d)+(f_{0}*\chi_{0})(b)f_{1}(d),\forall b,d\in SU_{q}^{0}(2)
$$
$$
(\chi_{0}*f_{0})(bd)=f_{0}(b)(\chi_{0}*f_{0})(d)+(\chi_{0}*f_{0})(b)f_{1}(d),\forall b,d\in SU_{q}^{0}(2).
$$

We now calculate:
$$
\begin{tabular}{|r  l|}
  \hline
&   \\
$(f_{0}*\chi_{0})(\alpha)=0$ & $(f_{0}*\chi_{0})(\alpha^{*})=0 $ \\
$(f_{0}*\chi_{0})(\gamma)=0$ & $(f_{0}*\chi_{0})(\gamma^{*})=-q^{-2}$\\
& \\
\hline
&\\
$(\chi_{0}*f_{0})(\alpha)=0$ & $(\chi_{0}*f_{0})(\alpha^{*})=0 $ \\
$(\chi_{0}*f_{0})(\gamma)=0$ & $(\chi_{0}*f_{0})(\gamma^{*})=-1$\\
&   \\
  \hline
\end{tabular}
$$
proving (3).
\\

To prove $(4),$ we	first calculate
$$(\chi_{k}*\chi_{j})(bd)=\chi_{k}(\chi_{j}*(bd))=\chi_{k}(b(\chi_{j}*d)+(\chi_{j}*b)(f_{j}*d))=$$
$$
\chi_{k}(b(\chi_{j}*d))+\chi_{k}((\chi_{j}*b)(f_{j}*d))=
$$
\begin{equation}\label{con}
\epsilon(b)(\chi_{k}*\chi_{j})(d)+\chi_{k}(b)(f_{k}*\chi_{j})(d)+\chi_{j}(b)(\chi_{k}*f_{j})(d)+(\chi_{k}*\chi_{j})(b)(f_{k}*f_{j})(d)
\end{equation}

(here we used that the counit $\epsilon$ is the identity in the convolution product). If we call the left side of $4(a)$ by $X',$ we see that if we use~\eqref{con} and $f_{0}*f_{2}=f_{1}$ (as $f_{0}=f_{2}$) on $X'(bd)$ that
$$
X'(bd)=\epsilon(b)X'(d)+X'(b)f_{1}(d)+
$$
$$
\chi_{2}(b)(q (f_{2}*\chi_{0})(d)-q^{-1}(\chi_{0}*f_{2})(d))+\chi_{0}(b)(q(\chi_{2}*f_{0})(d)-q^{-1} (f_{0}*\chi_{2})(d))
$$ 
and from $3(a),3(b)$ we see that the last two terms disappears. A calculation gives that $X'(\alpha)=1,X'(\alpha^{*})=-q^{2}$ and $=0$ on $\gamma,\gamma^{*},$ proving $4(a).$ To prove $4(b),$ we substitute the formula $\chi_{1}=\frac{q^{2}}{1-q^{2}}(f_{1}-\epsilon)$ into the left side to get 
$$
q^{2}\chi_{1}*\chi_{0}-q^{-2}\chi_{0}*\chi_{1}=\frac{q^{2}}{1-q^{2}}(q^{-2}\chi_{0}*\epsilon-q^{2}\epsilon*\chi_{0}+(q^{2}f_{1}*\chi_{0}-q^{-2}\chi_{0}*f_{1}))=
$$
$$
\frac{q^{2}}{1-q^{2}}(q^{-2}-q^{2})\chi_{0}=(1+q^{2})\chi_{0}
$$
by $3(c).$ The last equality can now be proved in a similar way or alternatively by applying the $*$-involution to $4(b)$ and use $2(a),2(b).$
\end{proof}
The idea is now is to use the linear functionals $f_{0},f_{1}, f_{2}, \chi_{0} ,\chi_{1}, \chi_{2}$ to create a $SU_{q}^{0}(2)$-bimodule $\Gamma$ and a linear map $d:SU_{q}^{0}(2)\rightarrow \Gamma$ which is a derivation i.e $d(bc)=bd(c)+d(b)c,\forall b,c\in SU_{q}^{0}(2).$ Once we have proven some further properties of these constructions, we will see that these are tools we can use to completely classify the finite dimensional representations!
\\

Lets get to it.
\begin{defn}
Let $\Gamma$ be the free left $SU_{q}^{0}(2)$-module generated by three elements $\omega_{0},\omega_{1},\omega_{2}.$
So every element in $\Gamma$ is of the form $\sum_{k=0}^{2}b_{k}\omega_{k}$ with $b_{k}\in SU_{q}^{0}(2)$ and the left multiplication by $SU_{q}^{0}(2)$ is given by
$$b\cdot(\sum_{k=0}^{2}b_{k}\omega_{k})=\sum_{k=0}^{2}b b_{k}\omega_{k}.$$
We now define a right $SU_{q}^{0}(2)$-multiplication on $\Gamma$ by
$$(\sum_{k=0}^{2}b_{k}\omega_{k})\cdot b=\sum_{k=0}^{2}b_{k}(f_{k}*b)\omega_{k}.$$
Since each $f_{k}$ is a homomorphism from $SU_{q}^{0}(2)$ to $\mathbb{C},$ we have $(f_{k}*b)(f_{k}*c)=f_{k}*(bc)$ and so we get a well defined right $SU_{q}^{0}(2)$-module structure. Furthermore, it is easy to see that $b\cdot (x\cdot c)=(b\cdot x)\cdot c$ for $b,c\in SU_{q}^{0}(2)$ and $x\in\Gamma.$
\end{defn}

\begin{defn}
Let $d:SU_{q}^{0}(2)\rightarrow \Gamma$ be the linear map $b\mapsto \sum_{k=0}^{2}(\chi_{k}*b)\omega_{k}.$
\end{defn}
\begin{lem}
$d:SU_{q}^{0}(2)\rightarrow \Gamma$ is a derivation.
\end{lem}
\begin{proof}
Recall from Proposition~\ref{conprop} that convolution by $\chi_{k}$ satisfies 
\begin{equation}\label{deri}
\chi_{k}*(bc)=b(\chi_{k}*c)+(\chi_{k}*b)(f_{k}*c).
\end{equation}
It now follows from the definition of the right and left $SU_{q}^{0}(2)$-multiplication that we have $d(bc)=bd(c)+d(b)c.$
\end{proof}
\begin{prop}\label{fu}
The only elements in kernel of $d$ is of the form $\lambda I$ with $\lambda\in\mathbb{C}.$
\end{prop}
\begin{proof}
We need to prove that for any $b\in SU_{q}^{0}(2),$ we have that if $\chi_{k}*b=0$ for $k=0,1,2,$ then $b=\lambda I$ for some $\lambda\in\mathbb{C}.$ Let $B=\{x\in SU_{q}^{0}(2):\chi_{k}*x=0,k=0,1,2\}.$  Notice that for any $h\in (SU_{q}^{0}(2))^{*}$ we have $\chi_{k}*(x*h)=(\chi_{k}*x)*h,$ so if $\chi_{k}*x=0$ then the same is true for $x*h.$ Hence $B$ is closed under right convolution with any element in $(SU_{q}^{0}(2))^{*}.$
\\

Recall from Lemma~\ref{bar} that we have a vector space basis for $SU_{q}^{0}(2)$ given by the elements
$$
A_{k,n,m}=
\left\{\begin{array}{ccc}
 \alpha^{k}\gamma^{*n}\gamma^{m}, & \text{If $k\geq 0$}\\
\alpha^{*(-k)}\gamma^{*n}\gamma^{m}, & \text{If $k<0$}
\end{array}\right\}.
$$
Assume we have $b=\sum_{k,n,m} \beta_{k,n,m}A_{k,n,m}$ such that $\chi_{k}*b=0$ for $k=0,1,2.$ Let further assume that $b$ is minimal in the sense that no set of basis vectors can be removed without making the convolution non-zero (except removing all of them of course!). As $\chi_{1}=\frac{q}{1-q^{2}}(f_{1}-\epsilon),$ we get 
$$
0=\chi_{1}*b=\sum_{k,n,m} \frac{q^{2}}{1-q^{2}}(q^{2(n-k-m)}-1) \beta_{k,n,m}A_{k,n,m}
$$
and so the only possible non-zero indices $(k,n,m)$ must satisfy $n=k+m.$ From $f_{0}*\chi_{0}=q^{2}\chi_{0}*f_{0}$ we find $\chi_{0}*(f_{0}*x)=(\chi_{0}*f_{0})*x=q^{2}(f_{0}*\chi_{0})*x=q^{2}f_{0}*(\chi_{0}*x)=0$ for any $x\in B,$ and similar is true for both $\chi_{1}$ and $\chi_{2}.$ Hence $B$ is closed under left convolution by $f_{0}.$ We now consider the operator $f(x)$ given by $x\mapsto f_{0}*x*f_{0}.$ It is easy to see that $f$ is an algebra endomorphism of $SU_{2}^{0}(2)$ and that 
\begin{equation}
\begin{array}{cccc}
f(\alpha)=q^{-2}\alpha & f(\gamma)=\gamma & f(\alpha^{*})=q^{2}\alpha^{*} & f(\gamma^{*})=\gamma^{*}.\\
\end{array}
\end{equation}
By the earlier calculations, $B$ is invariant under $f$ and from this we see that we can also assume that the nonzero coefficients $\beta_{k,n,m}$ has $k$ fixed. Let us now call them $\beta_{n.m}$ instead.
\\

Lets assume that $k\geq 0$ and consider $\chi_{0}*b$ (if $k<0$ consider $\chi_{2}*b$ instead in the calculation below). If we iterate~\eqref{deri}, we see that for a product $\prod_{j=0}^{k}c_{j}$ of elements in $SU_{q}^{0}(2)$ (assume that the product multiplies on the right side i.e $\prod_{j=0}^{k}c_{j}=c_{0}c_{1}...c_{k}$) we have
\begin{equation}\label{deriprod}
\chi_{k}*\left(\prod_{j=0}^{k}c_{j}\right)=\sum_{l=0}^{k}\left(\prod_{j=0}^{l-1}c_{j}\right)(\chi_{k}*c_{l})f_{k}*\left(\prod_{j=l+1}^{k}c_{j}\right).
\end{equation}
By using this and the table~\eqref{chicon}, we see that 
$$\chi_{0}*b=\sum_{n,m}\beta_{n,m}C_{n.m}A_{k+1,n-1,m}$$
where $C_{n,m}$ is $\neq 0$ if $n>0$ and $=0$ if $n=0.$ Thus the $A_{k,n,m}$'s in $b$ must have $n=0.$ From the earlier we know that $n=k+m.$ Hence $0=k+m$ and as these where positive integers, we must have $n=m=k=0.$
\end{proof}
\begin{defn}
Given a finite dimensional representation $U\in M_{n}(\mathbb{C})\otimes SU_{q}^{0}(2),$ we define three matrices $A^{U}_{0},A^{U}_{1},A^{U}_{2}\in M_{n}(\mathbb{C})$ by
\begin{equation}
\begin{array}{ccc}
A^{U}_{0}:=(\iota\otimes \chi_{0})U,& A^{U}_{1}:=(\iota\otimes \chi_{1})U,  & A^{U}_{2}:=(\iota\otimes \chi_{2})U. \\
\end{array}
\end{equation}
\end{defn}
\begin{lem}\label{system}
For any finite dimensional representation $U=(u_{j,k})_{j,k}\in M_{n}(\mathbb{C})\otimes SU_{q}^{0}(2) ,$ we have
\begin{enumerate}
\item
\begin{enumerate}
\item $q A^{U}_{2}A^{U}_{0}-q^{-1}A^{U}_{0}A^{U}_{2}=A^{U}_{1}$
\item $q^{2}A^{U}_{1}A^{U}_{0}-q^{-2}A^{U}_{0}A^{U}_{1}=(1+q^{2})A^{U}_{0}$
\item $q^{2}A^{U}_{2}A^{U}_{1}-q^{-2}A^{U}_{1}A^{U}_{2}=(1+q^{2})A^{U}_{2}$
\end{enumerate}
\item
If furthermore $U$ is unitary, then
\begin{enumerate}
\item
$-q (A^{U}_{0})^{*}=A^{U}_{2}$
\item
$(A^{U}_{1})^{*}=A^{U}_{1}$
\end{enumerate}
\end{enumerate}
\end{lem}
\begin{proof}
$(1)$ For any $g,h\in (SU_{q}^{0}(2))^{*},$ we have 
\begin{equation}\label{mult}
(\iota\otimes (g*h))U=(\iota\otimes g\otimes h)(\iota\otimes \Delta)U=(\iota\otimes g\otimes h)U_{1,2}U_{1,3}=$$
$$((\iota\otimes g)U)((\iota\otimes h)U).
\end{equation}
Now $(a),(b),(c)$ follows from~\eqref{mult} and Proposition~\ref{conprop}.
\\

$(2)$ If $U$ is unitary, then $(\iota\otimes S)U=U^{-1}=U^{*}$ and so for any $g\in(SU_{q}^{0}(2))^{*},$ we have
$$
((\iota\otimes g)U)^{*}=(\overline{g(u_{k,j})})_{j,k}=(\overline{g((S(u_{j,k}))^{*})})_{j,k}=(\iota\otimes g^{*})U.
$$
(2) from Proposition~\ref{conprop} now gives $(a),(b).$ 
\end{proof}
\begin{prop}\label{amatrix}
Let $U\in M_{n}(\mathbb{C})\otimes SU_{q}^{0}(2)$ and $V\in M_{m}(\mathbb{C})\otimes SU_{q}^{0}(2)$ be two finite dimensional representations and $T\in M_{n,m}(\mathbb{C})$.
The following is equivalent
\begin{enumerate}
\item
$A_{k}^{U}T=T A_{k}^{V}$ for $k=0,1,2.$
\item
$U(T\otimes I)=(T\otimes I)V$

\end{enumerate}
\end{prop}
\begin{proof}
The tensor product $I\otimes d:M_{n}(\mathbb{C})\otimes SU_{q}^{0}(2)\rightarrow M_{n}(\mathbb{C})\otimes \Gamma$ is still a derivation, as is easy to check. For a finite dimensional representation $W\in M_{j}(\mathbb{C})\otimes SU_{q}^{0}(2),$ let 

\begin{equation}
A^{W}:=\sum_{k=0}^{2}A^{W}_{k} \omega_{k}\in M_{j}(\mathbb{C})\otimes \Gamma.
\end{equation}
$(1)$ is of course now equivalent to $(T\otimes \iota)A^{W}=A^{W}(T\otimes \iota).$
We then get 
$$(I\otimes d)W=\sum_{k=0}^{2}(I\otimes \chi_{k}*)(W) \omega_{k}=\sum_{k=0}^{2}(I\otimes I \otimes\chi_{k})(I\otimes \Delta)(W) \omega_{k}= $$
$$
\sum_{k=0}^{2}(I\otimes I \otimes\chi_{k})(U_{1,2}U_{1,3}) \omega_{k}=W(\sum_{k=0}^{2}A^{W}_{k} \omega_{k})=W A^{W}.
$$
This calculations now gives
 $$0=(I\otimes d)(I)=(I\otimes d)(W^{*}W)=(I\otimes d)(W^{-1})W+W^{-1}(I\otimes d)(W)=$$
$$(I\otimes d)(W^{-1})W+A^{W}$$
and hence $(I\otimes d)(W^{-1})=-A^{W}W^{-1}.$
Notice that if $U(T\otimes I)V^{-1}=(R\otimes I)$ for some $T,R\in M_{n,m}(\mathbb{C}),$ then 
$$(I\otimes \epsilon)(U(T\otimes I)V^{-1})=(I\otimes \epsilon)(U)(T\otimes I)(I\otimes \epsilon)(V^{-1})=T\otimes I$$ since $(I\otimes \epsilon)W=I$ for any FDR $W.$ So $T=R$ and hence $U(T\otimes I)=(T\otimes I)V$ is equivalent to $U(T\otimes I)V^{-1}\in M_{n,m}(\mathbb{C}).$ We then get
$$
(I\otimes d)(U(T\otimes I)V^{-1})=U(A^{U}(T\otimes I)-(T\otimes I)A^{V})V^{-1}.
$$
By Proposition~\ref{fu}, we have $(I\otimes d)(U(T\otimes I)V^{-1})=0$ iff $U(T\otimes I)V^{-1}\in M_{n,m}(\mathbb{C}).$
This establishes the equivalence between $(1)$ and $(2).$
\end{proof}
If $U\in M_{n}(\mathbb{C})\otimes SU_{q}^{0}(2)$ is a unitary FDR and $U$ is not irreducible, then there exists an orthogonal projection $ P\in M_{n}(\mathbb{C})$ such that $(P\otimes I) U=U (P\otimes I).$ By Proposition~\ref{amatrix} we also have $P A_{k}^{U}=A_{k}^{U} P$ for $k=0,1,2$ and conversely any orthogonal projection $P$ commuting with the $A_{k}^{U}$'s determines an intertwiner for $U.$ We now analyze the situation a bit more carefully. Lets define a general system $\textbf{A}=\{A_{0},A_{1},A_{2}\}\subseteq M_{n}(\mathbb{C})$ such that
\begin{equation}\label{sys1}
\begin{array}{ccc}
q A_{2}A_{0}-q^{-1}A_{0}A_{2}=A_{1}\\
q^{2}A_{1}A_{0}-q^{-2}A_{0}A_{1}=(1+q^{2})A_{0}
\end{array}
\end{equation}
\begin{equation}\label{sys2}
\begin{array}{ccc}
-q A^{*}_{0}=A_{2}\\
A_{1}^{*}=A_{1}.
\end{array}
\end{equation}
Call such system $\textbf{A}$ an \textit{infinitesimal representation} of $SU_{q}^{2}$
We say that an infinitesimal representation $\textbf{A}$ is irreducible if there are no nontrivial idempotents in $M_{n}(\mathbb{C})$ that commutes with all the members of $\textbf{A}.$ From Proposition~\ref{amatrix}, we know that a system coming from a FDR $U$ is irreducible iff $U$ is irreducible. We now classify all irreducible $\textbf{A}.$
\begin{prop}
Let $\textbf{A}\subseteq M_{n+1}(\mathbb{C})$ be irreducible, then there is an orthonormal basis $\{\xi_{-n},\xi_{2-n},..,\xi_{n-2},\xi_{n}\}\subseteq \mathbb{C}^{n+1}$ such that in this basis $A_{0},A_{1},A_{2}$ acts as 
\begin{equation}
\begin{array}{cccc}
A_{0}\xi_{k}=-c_{k+1}\xi_{k+2} &A_{1}\xi_{k}=d_{k}\xi_{k} & A_{2}\xi_{k}=-q c_{k}\xi_{k-2} & \text{for all $k$,}\\
\end{array}
\end{equation}
where $c_{k}=\frac{\sqrt{(q^{-k}-q^{n})(q^{-n}-q^{2-k})}}{q^{-2}-1},$ $d_{k}=\frac{q^{-2k}-1}{q^{-2}-1}.$
In particular, the irreducible $\textbf{A}$'s are determined up to unitary equivalence by the dimension.
\end{prop}
\begin{proof}
The proof is quite easy but long, boring and more or less the same as the one given for $\su$ in the introduction. For those who are interested anyway, we refer to Woronowizs paper ("Twisted $SU(2)$ Group. An Example of a
Non-Commutative Differential Calculus", Theorem $5.4$).
\end{proof}

\begin{cor}
For every $n\in\mathbb{N},$ there are at most one irreducible representation of dimension $n$.
\end{cor}

We are now ready to show that our representations $U_{n}$ from Definition~\ref{un} is actually irreducible. One slight issue is that they are not unitary for any $n>0,$ for example, a calculation yield that 
$$U_{1}=\left[
\begin{array}{cc}
\alpha&   \gamma^{*}\\
-q\gamma & \alpha^{*}\\
\end{array}\right]
=\left[
\begin{array}{cc}
 q^{-\frac{1}{2}}&   0\\
0 & - q^{\frac{1}{2}}\\
\end{array}\right
]\left[
\begin{array}{cc}
\alpha&   -q\gamma^{*}\\
\gamma & \alpha^{*}\\
\end{array}\right]
\left[
\begin{array}{cc}
q^{\frac{1}{2}}&   0\\
0 & -q^{-\frac{1}{2}}\\
\end{array}\right
]
$$
which is not unitary. However, since every FDR $V\in M_{n}(\mathbb{C})\otimes SU_{q}^{0}(2)$ is actually equivalent to an unitary FDR in the sense that there exists a $T\in M_{n}(\mathbb{C})$ such that $(T^{-1}\otimes I)V(T\otimes I)$ is a unitary and then
$$
(I\otimes \chi_{k})((T^{-1}\otimes I)V(T\otimes I))=T^{-1}(I\otimes\chi_{k})(U) T=T^{-1}A^{V}_{k}T.
$$
We see that the system $A^{U}_{0},A^{U}_{1},A^{U}_{2}$ (which satisfies only~\eqref{sys1}) is equivalent to a system satisfying also~\eqref{sys2}. If $A^{U}_{0},A^{U}_{1},A^{U}_{2}$ is irreducible (in the sense that no nontrivial idempotents commutes with it), then obviously so is any equivalent system. Hence to prove that $U$ is irreducible, we only need to prove that the system $A^{U}_{0},A^{U}_{1},A^{U}_{2}$ is irreducible.

\begin{prop}
The FDR $U_{n}\in V_{n}\otimes SU^{0}_{q}(2)$ is irreducible. Furthermore, every irreducible FDR of $SU_{q}(2)$ is equivalent to one of the $U_{n}$'s. 
\end{prop}
\begin{proof}
We try to calculate $\chi_{0}*\alpha^{k}\gamma^{*(n-k)}.$ By formula~\eqref{deriprod}, for $0\leq k\leq n$
$$\sum_{j=1}^{k}\alpha^{j-1}\chi_{0}(\alpha)f_{0}*(\alpha^{k-j}\gamma^{*(m-k)})+\sum_{j=1}^{n-k}\alpha^{k}\gamma^{*(j-1)}\chi_{0}(\gamma^{*})f_{0}*(\gamma^{*(n-k-j)})=$$
$$
\sum_{j=1}^{n-k}\alpha^{k}\gamma^{*(j-1)}\chi_{0}(\gamma^{*})f_{0}*(\gamma^{*(n-k-j)})=\sum_{j=1}^{n-k}\alpha^{k}\gamma^{*(j-1)}(-q^{-1}\alpha)q^{n-k-j}\gamma^{*(n-k-j)}=
$$
$$
-\left(\sum_{j=1}^{n-k}q^{-(j-1)}(q^{-1})q^{n-k-j}\right)\alpha^{k+1}\gamma^{*(n-k-1)}=-q^{n-k-2}\left(\sum_{j=0}^{n-k-1}q^{-2j}\right)\alpha^{k+1}\gamma^{*(n-k-1)}=
$$
$$
-q^{n-k-2}\frac{1-q^{-2(n-k)}}{1-q^{-2}}\alpha^{k+1}\gamma^{*(n-k-1)}.
$$
As we also have $\chi_{0}*(\alpha^{k}\gamma^{*(n-k)})=\sum_{j=0}^{n}\alpha^{j}\gamma^{*(n-j)}\otimes \chi_{0}(u_{j,k}),$ the matrix $A_{0}^{U_{n}}$ has in the basis $\{\gamma^{*n},\alpha\gamma^{*(n-1)},...\alpha^{n-1}\gamma^{*},\alpha^{n}\}$  has the value $-q^{n-k-2}\frac{1-q^{-2(n-k)}}{1-q^{-2}}$ at the $(k+1,k)$ entry and the rest of the entries are zero. A similar calculation gives that an entry of $A_{2}^{U}$ is $\neq 0$ iff its index is in the form $(k,k+1).$ Finally, by using $\chi_{1}=\frac{q}{1-q^{2}}(f_{1}-\epsilon),$ we see that $A_{1}^{U_{n}}$ is diagonal with diagonal entries $\frac{q^{2}(q^{2(n-2k)}-1)}{1-q^{2}}.$ As $\frac{q^{2}(q^{2(n-2k)}-1)}{1-q^{2}}\neq\frac{q^{2}(q^{2(n-2j)}-1)}{1-q^{2}}$ for $k\neq j,$ the matrix $A^{U_{n}}_{1}$ separates the eigenvectors $\alpha^{k}\gamma^{*(n-k)},$ while $A^{U_{n}}_{0}$ maps $\alpha^{k}\gamma^{*(n-k)}$ to a nonzero multiple of $\alpha^{k+1}\gamma^{*(n-k-1)}$ for $0\leq k<n$ and $A_{2}^{U_{n}}$ maps $\alpha^{k}\gamma^{*(n-k)}$ to a nonzero multiple of $\alpha^{k-1}\gamma^{*(n-k+1)}$ for $0< k\leq n.$ It is now easy to see that no nontrivial subspace is invariant under all these matrices. This is equivalent to the irreducibility condition defined above.
\\

Since there is up to equivalence only one irreducible representation for each dimension $n,$ we get that any irreducible FDR of $SU^{0}_{q}(2)$ most be equivalent one of the $U_{n}$'s.
\\

There is now still the matter if there would be an irreducible FDR $V=(v_{j,k})$ of $SU_{q}(2)$ (the CQG) that was not equivalent to any one of the $U_{n}$'s. However, by taking tensor products of $U_{1},$ we see that the linear span of coefficients of FDR's of $SU_{q}^{0}(2)$ span $SU_{q}^{0}(2)$ and by the theory of CQG's this will imply that $h(v_{j,l}A_{k,m,n})=0$ for all $k\in \mathbb{Z},m,n\in\mathbb{N}$ (where $h$ is the Haar state on $SU_{q}(2)$) and hence $h(v_{j,l}v^{*}_{j,l})=0$ as $SU_{q}^{0}(2)$ is dense in $SU_{q}(2),$ giving us a contradiction.
\end{proof}

There is only the last claim in the statement of Theorem~\ref{woroz} left to prove, which is that a FDR $V$ is up to equivalent determined by its weight function. To this end, define $M_{V}(t)=\sum_{k=-\infty}^{\infty}M_{V}(k)e^{i k t}$ and consider $\sin(t)M_{V}(t).$ An easy calculation using~\eqref{wf} gives that $\sin(t)M_{(n)}(t)=\sin((n+1)t)$ and hence we can determine the multiplicity of $U_{(n)}$ in the decomposition of $V$ into irreducible factors by consider $\frac{1}{2\pi}\int_{-\pi}^{\pi}(\sin{t})^{2}M_{V}(t)M_{(n)}(t)dt.$
\end{subsection}


\begin{thebibliography}{99}
\bibitem{wor} {\sc S.L. Woronowizs}, {\it Twisted $SU(2)$ Group. An Example of a
Non-Commutative Differential Calculus}, {\rm {\rm Publ. Res. Inst. Math. Sci. \textbf{23} (1987), no. 1, 117–181}}
\end{thebibliography}
\end{document}